\documentclass[11pt]{amsart}

\usepackage{amsfonts, amssymb, amscd}

\usepackage{verbatim}
\usepackage{amssymb}
\usepackage{mathrsfs}
\usepackage{graphicx}
\usepackage{cite}
\usepackage[all]{xy}
\usepackage{tikz}
\usepackage{subfiles}
\usepackage[toc,page]{appendix}
\usepackage{mathtools}

\usepackage{hyperref}

\newcommand{\Pp}{\mathbb{P}}
\newcommand{\Rr}{\mathbb{R}}
\newcommand{\Qq}{\mathbb{Q}}
\newcommand{\Nn}{\mathbb{N}}

\newcommand{\Dd}{\mathcal{D}}

\newcommand{\Oo}{\mathcal{O}}
\newcommand{\Ii}{\mathcal{I}}
\newcommand{\Jj}{\mathcal{J}}
\newcommand{\Ee}{\mathcal{E}}

\newcommand{\Ttt}{\mathcal{T}}
\newcommand{\Tttt}{\mathcal{CT}}

\newtheorem{theorem}{Theorem}[section]
\newtheorem{lemma}[theorem]{Lemma}
\newtheorem{proposition}[theorem]{Proposition}
\newtheorem{definition}[theorem]{Definition}
\newtheorem{example}[theorem]{Example}
\newtheorem{corollary}[theorem]{Corollary}
\newtheorem{remark}[theorem]{Remark}

\begin{document}

\title{On Fujita's conjecture for pseudo-effective thresholds}

\begin{abstract}
We show Fujita's spectrum conjecture for $\epsilon$-log canonical pairs and Fujita's log spectrum conjecture for log canonical pairs. Then, we generalize the pseudo-effective threshold of a single divisor to multiple divisors and establish the analogous finiteness and the DCC properties.
\end{abstract}

\author{Jingjun Han}
\address{Beijing International Center for Mathematical Research, Peking University,
Beijing 100871, China} \email{hanjingjun@pku.edu.cn}

\author{Zhan Li}
\address{Beijing International Center for Mathematical Research, Peking University,
Beijing 100871, China} \email{lizhan@math.pku.edu.cn}

\maketitle

\tableofcontents

\section{Introduction}\label{sec: introduction}
Let $(X, \Delta)$ be a log pair with $X$ a normal projective variety over $\mathbb{C}$. Suppose $D$ is a big $\Rr$-Cartier divisor, then the \emph{pseudo-effective threshold} of $D$ with respect to $(X, \Delta)$ is defined by
\begin{equation}\label{eq: pseudo-effective threshold}
\tau(X, \Delta; D)\coloneqq \inf\{t \in \Rr_{\geq 0} \mid K_X +\Delta+ tD \ \text{is}\  \text{pseudo-effective}\}.
\end{equation} Recall that an $\Rr$-divisor is called \emph{pseudo-effective} if it is a limit of effective divisors in $N^1(X)_{\Rr}$.
\medskip


Fujita once made the following two conjectures \cite{Fujita92,Fujita96}.

\begin{theorem}[Fujita's spectrum conjecture, \cite{DiC17} Theorem 1.1]\label{thm:fujita}
 Let $n$ be a natural number, $S_n$ be the set of pseudo-effective thresholds $\tau(X,H):=\tau(X,\emptyset;H)$ of an ample divisor $H$ with respect to a smooth projective variety $X$ of dimension $n$. Then $S_n\cap[\epsilon,+\infty)$ is a finite set for any $\epsilon>0$.
 \end{theorem}

\medskip

\begin{theorem}[Fujita's log spectrum conjecture, \cite{DiC16} Theorem 1.2]\label{thm:fujitalog}
 Let $S_n^{ls}$ be the set of pseudo-effective thresholds $\tau(X,\Delta;H)$, where $X$ is a smooth projective variety of dimension $n$, $\Delta$ is a reduced divisor with simple normal crossing support, and $H$ is an ample Cartier divisor on $X$. Then $S_n^{ls}$ satisfies the ACC.
\end{theorem}

Fujita showed that his spectrum conjecture is a consequence of the minimal model program and the BAB (Borisov-Alexeev-Borisov) conjecture, \cite{Fujita96}. Recently, Di Cerbo studied these two problems. He proved Fujita's spectrum conjecture by using the special BAB conjecture, \cite{DiC17}, and proved Fujita's log spectrum conjecture by using the ACC for the lc thresholds, \cite{DiC16}.

\medskip
Recall that for a partially ordered set $(\mathcal S, \succeq)$, it is said to satisfy the DCC (\emph{descending chain condition}) if any non-increasing sequence $a_1 \succeq a_2 \succeq \cdots \succeq a_k \succeq \cdots$ in $S$ stabilizes. It is said to satisfy the ACC (\emph{ascending chain condition}) if any non-decreasing sequence in $S$ stabilizes. When $S$ is a set of real numbers, we consider the usual relation ``$\le$''.

\medskip

In this paper, we study Fujita's spectrum conjecture and Fujita's log spectrum conjecture in a more general setting, namely we allow the pair $(X,\Delta)$ to have singularities, and the coefficients of $\Delta$ and $H$ to vary in some fixed set.

\medskip

Fix a positive integer $n$, a positive real number $\epsilon$, and a subset $\Ii\subset[0,1]$. We will consider the following set.
\[
\begin{split}
\Ttt_{n,\epsilon}(\Ii)\coloneqq &\{\tau(X,\Delta;H)\mid \dim X=n, (X,\Delta) \text{~is~} \epsilon \text{-lc},\Delta\in \Ii,\\
&  H\text{ is a big and nef $\Qq$-Cartier Weil divisor}\}.
\end{split}
\]
We can now state one of the main results in this paper.
\begin{theorem}\label{thm: OneDCCfinite}
Let $n$ be a natural number, and $\epsilon$ be a positive real number. 
\begin{enumerate}
\item If $\Ii\subset[0,1]$ is a finite set, then $\Ttt_{n,\epsilon}(\Ii)\cap[\delta,+\infty)$ is a finite set for any $\delta>0$.
\item If $\Ii\subset[0,1]$ is a DCC set, then $\Ttt_{n,\epsilon}(\Ii)$ satisfies the ACC.
\end{enumerate}
\end{theorem}

Roughly speaking, we generalize Di Cerbo's result from a smooth variety with an ample divisor $H$ (Theorem \ref{thm:fujitalog} and Theorem \ref{thm:fujita}) to an $\epsilon$-lc pair with a big and nef divisor. Theorem \ref{thm: OneDCCfinite}(1) was a question asked by Di Cerbo (cf. \cite{DiC17} P.244).

\medskip

Our argument of the above result relies on the minimal model program and the recent progress on the BAB conjecture. In fact, we only need to use the special BAB conjecture to prove Theorem \ref{thm: OneDCCfinite} (see Theorem \ref{thm: BAB} and the remark below). In some sense, the $\epsilon$-lc condition is the weakest possible condition for Theorem \ref{thm: OneDCCfinite} to hold. If we relax the singularities from $\epsilon$-lc to klt, Theorem \ref{thm: OneDCCfinite} is no longer true, see Example \ref{ex:counter} for more details. However, if $H$ is assumed to be an ample Cartier divisor in Theorem \ref{thm: OneDCCfinite} (2), then the condition ``$(X,\Delta)$ is lc'' is enough. 


\begin{theorem}\label{thm: OneDCC}
Let $n$ be a natural number and $\Ii$ a DCC set of nonnegative real numbers. Let $\Tttt_{n}(\Ii)$ be the set of pseudo-effective threshold $\tau(X,\Delta;H)$ which satisfies the following conditions.
\begin{enumerate}
\item $X$ is a normal projective variety of dimension $n$,
\item $(X, \Delta)$ is lc and the coefficients of $\Delta$ are in $\Ii$, and
\item $H=\sum \mu_iH_i$, where $H_i$ is a nef Cartier divisor for each $i$, $\mu_i\in\Ii$, and $H$ is big.
\end{enumerate}
Then $\Tttt_{n}(\Ii)$ satisfies the ACC.
\end{theorem}
The version where $(X,\Delta)$ is log smooth, $\Delta$ is reduced, and $H$ itself is an ample divisor (Theorem \ref{thm:fujitalog}) was proven by Di Cerbo by the ACC for the log canonical thresholds and Global ACC for the lc pairs. Roughly speaking, the main difficulty in our setting comes from the fact that $(X,\Delta+\tau(X,\Delta;H)H)$ may not be log canonical, and we can not apply Global ACC directly. Instead, we use Global ACC for generalized polarized pairs, \cite{BZ16}, which generalizes Global ACC for the lc pairs, \cite{HMX14}, to the generalized lc pairs. The notation and the theory of generalized polarized pairs were introduced and developed in \cite{BZ16,Bir16a}. Furthermore, we can show a slightly stronger version of Theorem \ref{thm: OneDCC} by using an effective birationality result for generalized polarized pairs of general type established in \cite{BZ16}, see Theorem \ref{thm: OneDCC2}.


\medskip



It is natural to generalize the pseudo-effective threshold of a single divisor to multiple divisors. Let $(X, \Delta)$ be a log pair and $D_1,\ldots,D_m$ be big $\Rr$-Cartier divisors on $X$. The \emph{pseudo-effective polytope} (PE-polytope) of $D_1,\ldots,D_m$ with respect to $X$ is defined as
\begin{equation}\label{eq: PE-polytope}
\begin{split}
P(X, \Delta;D_1,\ldots,& D_m)\coloneqq \{(t_1,\ldots,t_m) \in\mathbb{R}_{\geq 0}^m \mid \\
&(K_X+\Delta+t_1D_1+\cdots+t_mD_m) ~\text{is}~\text{pseudo-effective}\}.
\end{split}
\end{equation}

We will show in Proposition \ref{prop: polytope} that $P(X, \Delta;D_1,\ldots,D_m)$ is indeed an unbounded polytope if $(X, \Delta)$ has klt singularities and $D_i$ is a big and nef $\Qq$-Weil divisor for each $i$. For convenience, we include all $(t_1,\ldots,t_m) \in\mathbb{R}_{\geq 0}^m$ such that $K_X+\Delta+\sum_{i=1}^mt_iD_i$ is pseudo-effective in $P(X, \Delta;D_1,\ldots,D_m)$, though the \emph{thresholds} only happen on the boundary. In particular, by comparing with \eqref{eq: pseudo-effective threshold}, we see that $P(X, \Delta;D)$ is the interval $[\tau(X, \Delta; D),+\infty)$.

\medskip

Fix two positive integers $n,m$, a subset $\Ii\subset[0,1]$, a positive real number $\epsilon$, and a nonnegative real number $\delta$, we consider the following set of truncated PE-polytopes.
\[
\begin{split}
{\mathcal P}_{n,m,\Ii,\epsilon,\delta}\coloneqq &\{P(X,\Delta;H_1,\ldots,H_m)\cap [\delta,+\infty)^m \mid \dim X=n, (X,\Delta) \text{~is~} \epsilon \text{-lc},\\
& \Delta\in \Ii, H_i\text{ is a big and nef $\Qq$-Cartier Weil divisor for each } i\}.
\end{split}
\]
For simplicity, if $\delta=0$, we denote ${\mathcal P}_{n,m,\Ii,\epsilon}\coloneqq{\mathcal P}_{n,m,\Ii,\epsilon,\delta}$. We will show the following results for PE-polytopes.

\begin{theorem}\label{thm: fujita}
Let $n,m$ be two natural numbers and $\epsilon, \delta$ be two positive real numbers, and $\Ii\subseteq [0,1]$ be a finite set. Then the set of truncated PE-polytope ${\mathcal P}_{n,m,\Ii,\epsilon,\delta}$ is finite.
\end{theorem}
Letting $m=1$ in Theorem \ref{thm: fujita}, we get Theorem \ref{thm: OneDCCfinite} (1).
\medskip


\begin{theorem}\label{thm: DCC}
Let $n,m$ be two natural numbers ,$\epsilon$ be a positive real number, and $\Ii\subseteq [0,1]$ a DCC set. 
Then the set of PE-polytopes ${\mathcal P}_{n,m,\Ii,\epsilon}$ is a DCC set under the inclusion relation ``$\supseteq$''.
\end{theorem}
Letting $m=1$ in Theorem \ref{thm: fujita}, we get Theorem \ref{thm: OneDCCfinite} (2).
\medskip

We note that we can not apply Theorem \ref{thm: OneDCCfinite} to prove Theorem \ref{thm: fujita} and Theorem \ref{thm: DCC} directly. For example, suppose we consider two testing divisors, and thus $P:=P(X,\Delta;H_1,H_2)$ is a two dimensional polytope (suppose it is non-degenerate). It is possible to construct a sequence of strictly decreasing sequence of convex polytopes $\{P_i\}_{i\in\Nn}$, such that  $\{P_i\}_{i\in\Nn}$ stabilizes along any vertical line, any horizontal line, and any line passing through the origin. For example, in Figure \ref{fig: example}, the sequence of polytopes $\{P_i\}_{i\in\Nn}$ is not stable near the point $\tau$. It is our hope that Theorem \ref{thm: fujita} and Theorem \ref{thm: DCC} would give more information on the testing divisors.
\begin{figure}[ht]
      \centering
\begin{tikzpicture}

\draw [help lines,<->] (0.0,3.2)--(0.0,0)--(3.2,0);
\draw [dashed](0.2,0)--(0.2,3.2);
\draw [dashed](0.8,0)--(0.8,3.2);
\draw [dashed](1.5,0)--(1.5,3.2);
\draw (0.1,3)--(0.9,0.8);
\draw [thick] (0.9,0.8)--(3.1,0);
\draw [thick](0.5,1.9)--(2.0,0.4);
\draw [thick](0.3, 2.45)--(1.25, 1.15);
\draw [help lines,->] (0.4,2.6)--(-0.4,2.6);
\draw [help lines,->] (0.5,2.1)--(-0.4,2.1);
\draw [help lines,->] (0.7,1.6)--(-0.4,1.6);

\node[left] at (-0.4,2.6) {\footnotesize$P_3$};
\node[left] at (-0.4,2.1) {\footnotesize$P_2$};
\node[left] at (-0.4,1.6) {\footnotesize$P_1$};
\node [below] at (0.3,3.2) {\footnotesize$\tau$};
\end{tikzpicture}
\caption{}
\label{fig: example}
\end{figure}

\medskip

As a corollary of Theorem \ref{thm: OneDCCfinite}(1), by the same argument of \cite{HJ16}, we can improve their main result for big and nef divisors (rather than big and semiample). This corollary was firstly proven by \cite{LTT} assuming a weak version of the BAB conjecture.

Recall that for smooth variety $X$, and a big $\Rr$-Cartier divisor $L$ on $X$, the $a$-constant is defined by $a(X,L):=\tau(X,L)$. For a singular projective variety $X$, the $a$-constant is defined by $a(X,L):=a(Y,\pi^{*}L)$, where $\pi: Y\to X$ is any log resolution of $(X,L)$.

\begin{corollary}[\cite{LTT} Theorem 4.10]\label{cor:LTT} Let $X$ be a smooth uniruled projective variety and $L$ a big and nef $\Qq$-divisor on $X$. Then there exists a proper closed subset $W\subset X$ such that every subvariety $Y$ satisfying $a(Y,L)>a(X,L)$ is contained in $W$.
\end{corollary}

Recall that we only use the special BAB conjecture to prove Theorem \ref{thm: OneDCCfinite} (1), thus yields another proof of the above corollary.
\medskip

Finally, it is reasonable to propose the following conjecture for PE-polytopes, which was proven in Theorem \ref{thm: OneDCC} for a single divisor.
\medskip

\textbf{DCC for PE-polytopes.} Let $n,m$ be two natural numbers, $\Ii$ be a DCC set of nonnegative real numbers. Then the set of PE-polytopes
\[
\begin{split}
{\mathcal P}_{n,m,\Ii}\coloneqq &\{P(X,\Delta;H_1,\ldots,H_m)\mid \dim X=n, (X,\Delta) \text{~is~} \text{lc},\\
& \Delta\in \Ii, H_i\text{ is a big and nef Cartier divisor for each }i\}
\end{split}
\]
satisfies the DCC under the inclusion of polytopes.
\medskip


\noindent\textbf{Acknowledgements}.
J.H. started this work during his visiting at Princeton University from Sep. 2015--Sep. 2016 under the supervision of J\'{a}nos Koll\'{a}r, and with the support of China Scholarship Council (CSC) and ``Training, Research and Motion'' (TRAM) network. He wishes to thank them all. We thank Chenyang Xu for giving not only the initial inspiration on this research topic, but also many valuable comments on a previous draft. We thank Gabriele Di Cerbo for useful E-mail correspondences, and Chen Jiang for helpful discussions. J.H. would like to thank his advisors Gang Tian and Chenyang Xu in particular for constant support and encouragement. This work is partially supported by NSFC Grant No.11601015 and the Postdoctoral Grant No.2016M591000.

\section{Preliminaries}\label{sec: preliminaries}
\subsection{Singularities} For basic definitions of log discrepancies and log canonical (lc), divisorially log terminal (dlt), kawamata log terminal (klt) singularities, we refer to \cite{KM98}. Recall that a log pair $(X,\Delta)$ is $\epsilon$-lc for some $\epsilon \geq 0$, if its minimal log discrepancy is greater or equal to $\epsilon$. 

For reader's convenience, we state the following lemma which is known as dlt modifications. 
\begin{lemma}[dlt modifications, c.f. \cite{HMX14} Proposition 3.3.1]
Let $(X,\Delta)$ be a lc pair. There there is a proper birational morphism $f:X'\to X$ with reduced exceptional divisors $E_i$, such that
\begin{enumerate}
\item $(X',\Delta'):=(X',f_{*}^{-1}\Delta+\sum E_i)$ is dlt,
\item $X'$ is $\Qq$-factorial,
\item $K_{X'}+\Delta'=f^{*}(K_X+\Delta)$.
\end{enumerate}
\noindent In particular, if $(X,\Delta)$ is klt, then $f$ is small.
\end{lemma}

We can apply dlt modifications to reduce the study of PE-polytopes from lc (resp. klt) pairs to $\Qq$-factorial dlt (resp. klt) pairs.
\begin{lemma}\label{lem:qfactorial}
  Let $(X,\Delta)$ be a lc pair, and $H_i$ be a big $\Rr$-Cartier divisor for any $1\le i\le m$. 
  If $f:(X',\Delta')\to (X,\Delta)$ is a dlt modification of $(X,\Delta)$, then
  $$P(X,\Delta; H_1,\ldots,H_m)=P(X',\Delta'; H'_1,\ldots,H'_m),$$
  where $H'_i=f^{*}H_i$.
\end{lemma}
\begin{proof}
  On one hand, let $(t_1,\ldots,t_m)\in P(X,\Delta; H_1,\ldots,H_m)$, then $K_X+\Delta+\sum_{i=1}^m t_iH_i$ is pseudo-effective. We have $K_{X'}+\Delta'+\sum_{i=1}^m t_iH_i'=f^{*}(K_X+\Delta+\sum_{i=1}^m t_iH_i)$ is also pseudo-effective.

  On the other hand, let $(t_1,\ldots,t_m)\in P(X',\Delta'; H'_1,\ldots,H'_m)$, $K_{X'}+\Delta'+\sum_{i=1}^m t_iH'_i$ is pseudo-effective. We have $K_X+\Delta+\sum_{i=1}^m t_iH_i=f_{*}(K_{X'}+\Delta'+\sum_{i=1}^m t_iH_i')$ is pseudo-effective.
\end{proof}

\subsection{Boundedness of Fano varieties}
A set of varieties $\mathcal{X}$ is said to form a \emph{bounded family} if there is a projective morphism of schemes $g: W\to T$, with $T$ of finite type, such that for every $X\in\mathcal{X}$, there is a closed point $t\in T$ and an isomorphism $W_t \simeq X$, where $W_t$ is the fibre of $g$ at $t$. A variety $X$ is called \emph{Fano} if it is lc and $-K_X$ is ample. The following result is a variant of the conjecture of Borisov-Alexeev-Borisov. 

\begin{theorem}[\cite{Bir16b} Corollary 1.2]\label{thm: BAB}
Let $n$ be a natural number and $\epsilon$ a positive real number. Then the projective varsities $X$ such that
\begin{enumerate}
\item $(X,\Delta)$ is $\epsilon$-lc of dimension $n$ for some boundary $\Delta$, and
\item $K_X+\Delta\equiv_{\Rr} 0$ and $\Delta$ is big,
\end{enumerate}
form a bounded family.
\end{theorem}
\begin{remark}
The special BAB conjecture, which assumes that the coefficients of $\Delta$ are more than or equal to a positive real number $\delta$ in Theorem \ref{thm: BAB}, was firstly proven in \cite{Bir16a}. In order to show Theorem \ref{thm: OneDCCfinite} and Theorem \ref{thm: fujita}, we only need to apply the special BAB conjecture.
\end{remark}

\subsection{Generalized polarized pairs}
The theory of generalized pair was developed in \cite{BZ16}.
\begin{definition}[generalized polarized pairs]
  A generalized polarized pair consists of a normal variety $X$ equipped with projective morphisms
\[
W \xrightarrow{f} X \to Z
\]
where $f$ is birational, $W$ is normal, an $\Rr$-boundary $\Delta\ge0$, and an $\Rr$-Cartier divisor $H_W$ on $W$ which is nef over $Z$ such that $K_X+\Delta+H$ is $\Rr$-Cartier with $H=f_{*}H_W$. We call $\Delta$ the boundary part and $H_W$ the nef part. We usually refer to the pair by saying $(X,\Delta+H)$ is a generalized pair with data $\xymatrix@C=0.5cm{W \ar[r]^{f} & X \ar[r] & Z}$ and $H_W$.
\end{definition}
In this paper, we only need to use the case that $Z$ is a point. Thus, we will drop $Z$, and say the pair is projective. Note that if $W'\to W$ is a projective birational morphism from a normal variety, then there is no harm in replacing $W$ with $W'$ and replacing $H_W$ with its pullback to $W'$.
\begin{definition}[generalized lc]
  Let $(X,\Delta+H)$ be a generalized polarized pair, which comes with the data $\xymatrix@C=0.5cm{W \ar[r]^{f} & X \ar[r] & Z},$ and replacing $W$, we may assume that $f$ is a log resolution of $(X,\Delta)$. We can write
  $$K_W+\Delta_W+H_W=f^{*}(K_X+\Delta+H)$$
  for some uniquely determined $\Delta_W$. 
  We say $(X,\Delta+H)$ is {\emph generalized lc} if every coefficient of $\Delta_W$ is less than or equal to 1.
\end{definition}
\begin{remark}[\cite{BZ16} Remark 4.2(6)]Let $(X,\Delta+H)$ be a generalized projective pair with data $\xymatrix@C=0.5cm{W \ar[r]^{f} & X}$ and $H_W$. We may assume that $f$ is a log resolution of $(X,\Delta)$. Assume that there is a contraction $X\to Y$. Let $F$ be a general fibre of $W\to Y$, $T$ the corresponding fibre of $X\to Y$, and $g: F\to T$ the induced morphism. Let
$$\Delta_{F}=\Delta_W|_{F}, H_{F}=H_W|_{F},\Delta_T=g_{*}(\Delta_{F}),H_{T}=g_{*}H_{F}.$$
Then $(T,\Delta_T+H_{T})$ is a generalized polarized projective pair with the data $\xymatrix@C=0.5cm{F \ar[r]^{g} & T}$ and $H_{F}$. Moreover,
$$K_T+\Delta_T+H_T=(K_X+\Delta+H)|_{T}.$$
In addition, $\Delta_{T}=\Delta|_{T}$ and $H_{T}=H|_{T}$.
\end{remark}
For more properties of generalized polarized pairs, we refer to \cite{BZ16,Bir16a}. We will need the following result to prove Theorem \ref{thm: OneDCC}.
\begin{theorem}[Global ACC for generalized pairs, \cite{BZ16} Theorem 1.6]\label{thm: genACC}
Let $\Ii$ be a DCC set of nonnegative real numbers and $n$ a natural number. Then there is a finite subset $\Ii_0\subseteq \Ii$ depending only on $\Ii,n$ such that
\begin{enumerate}
\item $X$ is a normal projective variety of dimension n,
\item $(X,\Delta+H)$ is generalized lc with data $\xymatrix@C=0.5cm{W \ar[r]^{f} & X}$ and $H_W$,
\item $H_W=\sum\mu_j H_{j,W}$, where $H_{j,W}$ are nef Cartier divisors and $\mu_j\in\Ii$,
\item $\mu_j=0$ if $H_{j,W}\equiv 0$,
\item the coefficients of $\Delta$ belong to $\Ii$, and
\item $K_X+\Delta+H\equiv_{\Rr}0$,
\end{enumerate}
\noindent then the coefficients of $\Delta$ and $\mu_j$ belong to $\Ii_0$.
\end{theorem}
\section{Fujita's spectrum conjecture and Fujita's log spectrum conjecture}
\begin{proof}[Proof of Theorem \ref{thm: OneDCCfinite}] By Lemma \ref{lem:qfactorial}, we can assume that $X$ is $\Qq$-factorial. Let $\tau=\tau(X,\Delta;H)$, then $K_X+\Delta+\tau H$ is pseudo-effective but not big. We can assume that $K_X$ is not pseudo-effective.

\medskip
Since $H$ is big and nef, there is an effective divisor $E$, and ample $\Qq$-divisors $A_{k}$ such that $H\sim_{\Qq} A_{k}+\frac{E}{k}$, for any $k\gg 1$. Let $N$ be a natural number such that $(X,\Delta+\tau\frac{E}{N})$ is $\frac{\epsilon}{2}$-lc, and $N'$ a natural numbers such that $N'A_{N}$ is a very ample divisor. Let $A'\in |N'A_{N}|$ be a general very ample divisor, and $H':=\frac{1}{N'}A'+\frac{E}{N}\sim_{\Qq} H$, then 
\[
(X,\Gamma)\coloneqq (X, \Delta+\tau H')
\]
is still $\frac{\epsilon}{2}$-lc

\medskip

According to \cite{BCHM10}, we may run a $(K_X+\Gamma)$-MMP, $\phi: X\dashrightarrow Y$, such that $\phi$ is a birational map, $K_Y+\phi_{*}\Gamma$ is semiample and $(Y, \phi_{*}\Gamma)$ is still $\frac{\epsilon}{2}$-lc. As $\phi$ is $(K_X+\Gamma)$-negative, $K_Y+\phi_{*}\Gamma$ is not big as well. 

\medskip

 Now $K_Y+\phi_{*}\Gamma$ defines a contraction $f: Y\rightarrow Z$. Let $F$ be a general fiber of $f$, we have $\dim(F)>0$. Restricting to $F$, we get
\begin{equation}\label{eq: numerically trivial1}
K_{F}+\phi_{*}\Gamma|_{F}= K_F +\phi_{*}\Delta|_{F}+\tau H_{Y}|_F\equiv_{\Rr}0,
\end{equation} where $H_{Y}$ is the strict transform of $H'$ on $Y$. We note that $H_{Y}$ is big and $H_{Y}|_F$ is also big since it is the restriction of a big divisor on a general fiber.

\medskip

Since $K_F +\phi_{*}\Gamma|_{F}$ is $\frac{\epsilon}{2}$-lc, according to Theorem \ref{thm: BAB}, $F$ belongs to a bounded family. We may find a very ample Cartier divisor $M_{F}$ on $F$, so that $-K_F\cdot M_{F}^{\dim F-1}$ is bounded. Besides, as $H' \sim_\Qq H$, we have $\phi_*(H') \equiv_{\Qq} \phi_*(H)$, and the intersection number $d:=H_{Y}|_F \cdot M_{F}^{\dim F-1}=\phi_*(H)|_F \cdot M_{F}^{\dim F-1}$ is a positive integer. Let $\phi_{*}\Delta|_F=\sum_j a_j\Delta_{F,j}$, where $a_j\in \Ii$, and $\Delta_{F,j}$ is a Weil divisor. By intersecting \eqref{eq: numerically trivial1} with $M_{F}^{\dim F-1}$, we obtain an equation for $\tau$,
\begin{equation}\label{eq: relation of tau1}
\tau d+\sum_j a_j b_j=c,
\end{equation}
where $c = - K_F \cdot M_{F}^{\dim F-1}$ is a nonnegative integral with only finite possibilities, $b_j = \Delta_{F,j} \cdot M_{F}^{\dim F-1}$ are nonnegative integers.

\medskip

First, we prove Theorem \ref{thm: OneDCCfinite} (1). By assumption $\tau\ge \delta>0$ and $\Ii$ is finite, $d,b_j$ are bounded above and thus only have finite possibilities. Hence there are only finite possibilities for the $\tau$ in equation \eqref{eq: relation of tau1}.

\medskip

Next, we prove Theorem \ref{thm: OneDCCfinite} (2). Recall that $a_j \in \Ii$ which is a DCC set, and $b_j \in \Nn$, then the finite summations $\sum_j a_j b_j \in \sum \Ii$ still form a DCC set. Now
\begin{equation}\label{eq:relationoftau2}
  \tau=\frac{1}{d}(c-\sum_j a_j b_j).
\end{equation}
The right hand side of the equation (\ref{eq:relationoftau2}) belongs to an ACC set, and $\Ttt_{n,\epsilon}(\Ii)$ satisfies the ACC.
\end{proof}

\begin{proof}[Proof of Corollary \ref{cor:LTT}]
Replacing Theorem \ref{thm:fujita} by Theorem \ref{thm: OneDCCfinite} in the proof of \cite{HJ16}, we get Corollary \ref{cor:LTT}.
\end{proof}

In the following example, we see that Theorem \ref{thm: OneDCCfinite} (1) is no longer true even for an ample Cartier divisor if we relax the singularities from $\epsilon$-lc to klt, and Theorem \ref{thm: OneDCCfinite} (2) is no longer true even for an ample $\Qq$-Cartier Weil divisor if we relax the singularities from $\epsilon$-lc to klt. We thank Chen Jiang for providing us this example.

\medskip

\begin{example}\label{ex:counter}
Let $n$ be a natural number and $\Pp(1,1,n)$ be the weighted projective space. It is a toric variety with the lattice $N$ spanned by $\{(1,0), (\frac 1 n, \frac 1 n)\}$. Let $v_1=(1,0), v_2=(0,1)$ and $v_3=(-\frac 1 n, -\frac 1 n)$, then the fan is generated by maximal cones $\sigma_1= \langle v_1,v_2\rangle, \sigma_2= \langle v_2,v_3\rangle$ and $\sigma_3= \langle v_1,v_3\rangle$ (c.f. \cite{Ful93} P.35). Let $D_i$ be the toric invariant divisor corresponding to $v_i$. By choosing $(n,0)$ and $(0,n)$ in the dual lattice $N^\vee$, we see that $nD_1 \sim D_3 $ and $nD_2 \sim D_3$. Thus
\[
-K_{\Pp(1,1,n)} \sim_\Qq D_1+D_2+D_3 \sim_\Qq ( 1+\frac 2 n)D_3.
\] Moreover $D_3$ is an ample Cartier divisor as it is associated to the lattice $(0,0)\in N^\vee$ for $\sigma_1$, $(n,0)\in N^\vee$ for $\sigma_2$ and $(0,n)\in N^\vee$ for $\sigma_3$.  As $\Pp(1,1,n)$ has Picard number one, $( 1+\frac 2 n)$ is the pseudo-effective threshold of $D_3$. This gives a set of varieties whose pseudo-effective spectrum is infinite away from $0$. Notice that the minimal log discrepancy of $\Pp(1,1,n)$ is $\frac 2 n$ (c.f. \cite{Bor97}), hence it is a counterexample if we replace $\epsilon$-lc by klt in Theorem \ref{thm: OneDCCfinite} (1).

\medskip

Besides, $D_1$ is an integral divisor and
\[
-K_{\Pp(1,1,n)} \sim_\Qq D_1+D_2+D_3 \sim_\Qq (n+2)D_1.
\]
Hence the pseudo-effective threshold of $D_1$ is $n+2$. Thus we get a family whose pseudo-effective thresholds (with respect to $D_1$) are strictly increasing. This gives a counterexample if we replace $\epsilon$-lc by klt in Theorem \ref{thm: OneDCCfinite} (2).
\end{example}

\begin{proof}[Proof of Theorem \ref{thm: OneDCC}]
Let $\tau=\tau(X,\Delta,H)$. We first prove the theorem for the case that $(X,\Delta)$ is $\Qq$-factorial klt.

By assumption, $H$ is big and nef, as in the proof of Theorem \ref{thm: OneDCCfinite}, we can run a $(K_X+\Delta+\tau H)$-MMP, $\phi: X\dashrightarrow Y$, and reach a minimal model $(Y,\phi_{*}(\Delta+\tau H))$, on which $K_Y+\phi_{*}(\Delta+\tau H)$ is semiample defining a contraction $Y\to Z$.

Taking a common log resolution $p: W\rightarrow (X,\Delta)$, $q: W\rightarrow (Y,\phi_{*}\Delta)$, let $H_W=p^{*}H$. Then
$(X,\Delta+\tau H)$ is generalized lc, as $(X,\Delta)$ is lc. Since $p^{*}(K_X+\Delta+\tau H)\ge q^{*}(K_Y+\phi_{*}(\Delta+\tau H))$, $(Y,\phi_{*}(\Delta+\tau H))$ is also generalized lc. Let $F$ be a general fiber of $W\to Z$, and $T$ be the corresponding fiber of $Y\to Z$. Again, we have $\dim(T)>0$. By restricting to the general fiber $T$, $(T,\phi_{*}(\Delta+\tau H)|_{T})$ is generalized lc, and $K_T+\phi_{*}(\Delta+\tau H)|_{T}\equiv0$. Since $H_W$ is big, $H_W|_F$ is not numerically trivial, and there exists some component $q^{*}(H_j)|_F$ of $H_W|_F$ which is not numerically trivial. If $\{\tau\}$ forms a strictly increasing sequence, then $\{\mu_j\tau\}$ belongs to a DCC set. According to Theorem \ref{thm: genACC}, $\{\mu_j\tau\}$ belongs to a finite set, a contradiction. Therefore, $\{\tau\}$ belongs to an ACC set.

\medskip

 For the general case, according to Lemma \ref{lem:qfactorial}, we may assume that $(X,\Delta)$ is $\Qq$-factorial dlt. If the statement were not true, then there exists a sequence of lc pairs $(X^{(i)},\Delta^{(i)})$, and a big $\Rr$-Cartier $H^{(i)}$ on $X^{(i)}$ satisfying the assumption of Theorem \ref{thm: OneDCC}, but $\tau_i:=\tau(X^{(i)},\Delta^{(i)}, H^{(i)})$ is strictly increasing. For any $1\ge \epsilon\ge 0$, let $\tau_{i,\epsilon}:=\tau(X^{(i)},(1-\epsilon)\Delta^{(i)};H^{(i)})$. It is clear that $\tau_{i,\epsilon}\ge \tau_i$, and there exists a decreasing sequence, $\epsilon_i\to 0$, such that $\tau_{i+1}>\tau_{i,\epsilon_i}\ge \tau_i$. Let $\Jj=\{1-\epsilon_i\}$. Now $(X^{(i)},(1-\epsilon_i)\Delta^{(i)})$ is $\Qq$-factorial klt, the coefficients of $(1-\epsilon_i)\Delta^{(i)}$ belong to $\Ii\Jj$, which is a DCC set. However, the sequence $\{\tau_{i,\epsilon_i}\}_{i\in\Nn}$ is strictly increasing. This contradicts to the above $\Qq$-factorial case.
\end{proof}

Inspired by a private communication with Di Cerbo, we can prove a slightly stronger version of Theorem \ref{thm: OneDCC} by using an effective birationality result for generalized polarized pairs of general type established in \cite{BZ16}.

\begin{theorem}\label{thm: OneDCC2}
Let $n$ be a natural number and $\Ii$ a DCC set of nonnegative real numbers. Let $\Dd_{n}(\Ii)$ be the set of pseudo-effective threshold $\tau(X,\Delta;M)$ which satisfies the following conditions.
\begin{enumerate}
\item $X$ is a normal projective variety of dimension $n$,
\item $(X, \Delta)$ is lc and the coefficients of $\Delta$ are in $\Ii$,
\item $M=\sum \mu_iM_i$, where $M_i$ is a nef Cartier divisor for each $i$, $\mu_i\in\Ii$, and
\item $K_X+\Delta+M$ is big
\end{enumerate}
Then $\Dd_{n}(\Ii)$ satisfies the ACC.
\end{theorem}
\begin{proof}
 Otherwise, there exists a sequence $\{(X^{(i)},\Delta^{(i)},M^{(i)})\}$, such that $\alpha_i=\tau (X^{(i)},\Delta^{(i)};M^{(i)})$ is strictly increasing and $\lim_{i\to +\infty}\alpha_i= \alpha\le 1$. Then, $\tau (X^{(i)},\Delta^{(i)};\alpha M^{(i)})=\frac{\alpha_i}{\alpha }\to 1$, and $K_X^{(i)}+\Delta^{(i)}+\alpha M^{(i)}$ is big. Since the coefficients of $\alpha M^{i}$ belong to the DCC set $\Ii\cup\alpha \Ii$, by Theorem 8.2 in \cite{BZ16}, there exists a natural number $m$ depending only on $n$ and $\Ii\cup\alpha \Ii$, such that the linear system $|\lfloor m(K_X+\Delta)\rfloor+\sum \lfloor m\alpha\mu_j\rfloor M_j|$ defines a birational map. By Lemma 2.3.4 in \cite{HMX13}, $K_X+\Delta+(2n+1)(m(K_X+\Delta+\alpha M))$ is big. Since
 \begin{align*}
&K_X+\Delta+(2n+1)(m(K_X+\Delta+\alpha M))\\
\sim_{\Rr} &((2n+1)m+1)(K_X+\Delta+\frac{(2n+1)m\alpha }{(2n+1)m+1}M),
\end{align*}

  we have $\alpha_i=\tau(X^{(i)},\Delta^{(i)};M^{(i)})\le \frac{(2n+1)m}{(2n+1)m+1}\alpha$. This implies that
  $$\lim_{i\to +\infty}\alpha_i \le \frac{(2n+1)m}{(2n+1)m+1}\alpha<\alpha,$$
  a contradiction.
  \end{proof}
  \begin{remark}
  In Theorem \ref{thm: OneDCC}, we do not require that $M$ to be big, and Theorem \ref{thm: OneDCC2} implies Theorem \ref{thm: OneDCC} . In fact, let $\Ii$ be a DCC set, $(X,\Delta)$ be a projective $\Qq$-factorial dlt pair of dimension $n$, $H=\sum \mu_i H_i$ be a big and nef $\Rr$-Cartier divisor. Set $\delta=\min \{\Ii^{>0}\}$, where $\Delta,\mu_j\in \Ii$. One can show that $K_X+\Delta+\frac{(2n+1)}{\delta}H$ is also big (see, for example, \cite{Bir16a} Lemma 2.30). Let $M=\frac{(2n+1)}{\delta}H$, Theorem \ref{thm: OneDCC} follows from Lemma \ref{lem:qfactorial} and Theorem \ref{thm: OneDCC2}.
  \end{remark}

\section{The Finiteness and the DCC property for pseudo-effective polytopes}\label{sec: proof}
In this section, we prove Theorem \ref{thm: fujita} and Theorem \ref{thm: DCC}, which generalize Theorem \ref{thm: OneDCCfinite} from a single divisor to multiple divisors. The proofs are similar.
\subsection{Pseudo-effective polytopes}

Let $X$ be a normal projective variety, and $V$ be a finite dimensional affine subspace of the real vector space $\mathrm{WDiv}_{\Rr}(X)$ of Weil divisors on $X$. Fix an $\Rr$-divisor $A\ge0$ and define (see \cite{BCHM10} Definition 1.1.4)
\[
\Ee_{A}(V)=\{\Delta=A+B \mid B\in V,B\ge0, K_X+\Delta \text{ is lc and pseudo-effective}\}.
\]

Notice that $\Ee_{A}(V)$ is a compact set by the lc requirement. Under this notation, we have the following result.

\begin{theorem}[\cite{BCHM10} Corollary 1.1.5]\label{thm: BCHM polytope}
Let $X$ be a normal projective variety, $V$ be a finite dimensional affine subspace of $\mathrm{WDiv}_{\Rr}(X)$ which is defined over rationals. Suppose there is a divisor $\Delta_0 \in V$ such that $K_X+\Delta_0$ is klt. Let $A$ be a general ample $\Qq$-divisor, which has no components in common with any element of $V$. Then $\Ee_{A}(V)$ is a rational polytope.
\end{theorem}

We can deduce that PE-polytopes are indeed polytopes under suitable assumptions from Theorem \ref{thm: BCHM polytope}. Notice that there are no lc restrictions for PE-polytopes. See Figure \ref{fig: PE-polytope} for a PE-polytope of two divisors.

\begin{figure}[h]
      \centering
\begin{tikzpicture}

\draw [help lines,<->] (0,2.5)--(0,0)--(2.5,0);
\draw [thick] (0,2.3)--(0,1.8)--(0.3,1.1)--(1,0.3)--(1.8,0)--(2.3,0);

\fill[gray!50, nearly transparent] (2.3,2.3)--(0,2.3)--(0,1.8)--(0.3,1.1)--(1,0.3)--(1.8,0)--(2.3,0) -- cycle;

\node[right] at (2.5,0) {\footnotesize$t_1$};
\node[above] at (0,2.5) {\footnotesize$t_2$};
\node at (1.2,2){\tiny$P(X, \Delta; D_1, D_2)$};

\end{tikzpicture}
\caption{A PE-polytope of two divisors}
\label{fig: PE-polytope}
\end{figure}
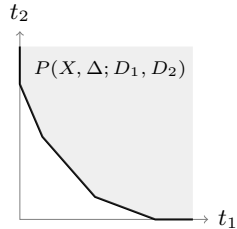

\begin{proposition}\label{prop: polytope}
Let $(X, \Delta)$ be a klt pair and $H_1, \ldots, H_m$ be big and nef $\Qq$-Weil divisors. Then
 $P(X,\Delta; H_1,\ldots,H_m)$ is an unbounded polytope.
\end{proposition}
\begin{proof}
By definition (see \eqref{eq: PE-polytope}),  $P(X,\Delta; H_1,\ldots,H_m)$ is convex. If $K_X$ is pseudo-effective, then $P(X,\Delta; H_1,\ldots,H_m) = \Rr_{\geq 0}^m$, hence we can assume that $K_X$ is not pseudo-effective. Then there exists a rational number $a>0$ such that $K_X+ a(H_1+ \cdots+H_m)$ is also not pseudo-effective. 

\medskip

Because $H_i$ are big divisors, there exists constant $\tau>0$, such that $K_X + \Delta + \sum_{i=1}^m t_i H_i$ is pseudo-effective whenever $t_i \geq \tau$. In other words, $\Rr^m_{\geq 0} \backslash [0, \tau]^m$ is contained in $P(X,\Delta; H_1,\ldots,H_m)$. Thus, it suffices to show that
$P(X,\Delta; H_1,\ldots,H_m)\cap [0, \tau]^m$ is a polytope.

\medskip

Recall that a divisor is pseudo-effective if and only if the divisor which is numerical to it is also pseudo-effective, thus,
\begin{equation}\label{eq: pseudo-effective is num}
P(X,\Delta; H_1,\ldots,H_m) = P(X,\Delta; H'_1,\ldots, H'_m)
\end{equation} for any $H_i \equiv_\Qq H_i'$. Since $H_i$ is big and nef, there is an effective divisors $E_i$, and ample $\Qq$-divisor $A_{i_k}$ such that $H_i\equiv A_{i_k}+\frac{E_i}{k}$, for any $k\gg 1$. Let $N$ be a natural number such that $(X,\Delta+\tau\sum_{i=1}^m\frac{E_i}{N})$ is klt, and $N_i>\tau$ be a natural number such that $N_iA_{i_N}$ is a very ample divisor. Let $A_i'\in |N_iA_{i_N}|$ be a general very ample divisor, and $H_i'=\frac{1}{N_i}A_i'+\frac{E_i}{N}$, $(X,\Delta+\tau\sum_{i=1}^m H'_i)$ is also klt. Now, we only need to show that $P(X,\Delta; H'_1,\ldots,H'_m)\cap [0, \tau]^m$ is a polytope.


\medskip

For each $i$, let $M_i'\in |A_i'|$ be another general very ample divisor, and let $M_i = \frac{a}{N_i}M'_i$. Let $V_i$ be the affine space $\Delta +\frac{a}{N}E_i+\sum_{i=1}^m \Rr H'_i$. By Theorem \ref{thm: BCHM polytope}, $\Ee_{M_i}(V_i)$ is a rational polytope.
Let $$\Ee_i:=\{(t_1,\ldots,t_{i-1},t_i+a,t_{i+1},\ldots,t_m)|\Delta+M_i+\frac{a}{N}E_i+\sum_{i=1}^m t_i H'_i\in\Ee_{M_i}(V_i)\}.$$
It is clear that $\Ee_i\cap [0, \tau]^m\subseteq P(X,\Delta; H'_1,\ldots,H'_m)\cap [0, \tau]^m$ for each $i$. Now if $(t_1, \ldots, t_m) \in P(X,\Delta; H'_1,\ldots,H'_m)\cap [0, \tau]^m$, there exists at least one $k$ such that $t_k \geq a$ (recall that we chose $a$ such that $K_X+ a(H_1+ \cdots+H_m)$ is not pseudo-effective). This implies that $\Delta+M_k+\frac{a}{N}E_k+\sum_{i=1}^m t_i H'_i-aH'_k\in\Ee_{M_k}(V_k)$, and $(t_1,\ldots,t_m)\in \Ee_k$. Thus,
$$\bigcup_{i=1}^m\Ee_i\cap [0, \tau]^m=P(X,\Delta; H'_1,\ldots,H'_m)\cap [0, \tau]^m.$$
In particular, $P(X,\Delta; H'_1,\ldots,H'_m)\cap [0, \tau]^m $ is a polytope by convexity.
\end{proof}
\subsection{Proofs of Theorem \ref{thm: fujita} and Theorem \ref{thm: DCC}}
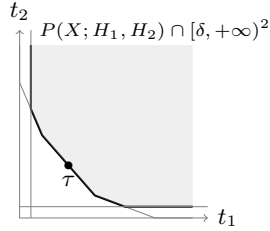
\begin{figure}[h]
      \centering
\begin{tikzpicture}

\draw [help lines,<->] (0,2.5)--(0,0)--(2.5,0);
\draw [help lines] (0,2.3)--(0,1.8)--(0.3,1.1)--(1,0.3)--(1.8,0)--(2.3,0);
\draw [thick] (.15, 2.3)--(0.15,1.45)--(0.3,1.1)--(1,0.3)--(1.4,0.15)--(2.3,0.15);
\draw [help lines] (0.15,0)--(.15, 2.5);
\draw [help lines] (0,.15)--(2.5,.15);
\fill[gray!50, nearly transparent] (2.3,2.3)--(.15, 2.3)--(0.15,1.45)--(0.3,1.1)--(1,0.3)--(1.4,0.15)--(2.3,0.15)-- cycle;

\node  at (.65,.7) {\tiny$\bullet$};
\node [below] at (.65,.7) {\footnotesize$\tau$};
\node[right] at (2.5,0) {\footnotesize$t_1$};
\node[above] at (0,2.5) {\footnotesize$t_2$};
\node at (1.8,2.5){\tiny$P(X; H_1, H_2) \cap [\delta,+\infty)^2$};

\end{tikzpicture}
\caption{A truncated PE-polytope}
\label{fig: a truncated PE-polytope}
\end{figure}

\begin{proof}[Proof of Theorem \ref{thm: fujita}] By Lemma \ref{lem:qfactorial}, we can assume that $X$ is $\Qq$-factorial. Choose arbitrary $\tau=(\tau_1\ldots,\tau_m)$ on the boundary of the truncated PE-polytope $P(X,\Delta;H_1,\ldots,H_m)\cap [\delta,+\infty)^m$ and in the interior of $[\delta,+\infty)^m$, then $K_X+\Delta+\tau_1H_1+ \cdots \tau_m H_m$ is pseudo-effective but not big. If $K_X$ is pseudo-effective, then $P(X,\Delta;H_1,\ldots,H_m)\cap [\delta,+\infty)^m = [\delta,+\infty)^m$, we can assume that $K_X$ is not pseudo-effective.

\medskip
\medskip

Since $H_i$ is big and nef, by the same argument in the proof of Theorem \ref{thm: OneDCCfinite}, we can find $\Qq$-divisors $H'_i\sim_{\Qq}H_i$, such that
$$(X,\Gamma)\coloneqq(X,\Delta+\sum_{i=1}^m \tau_iH'_i)$$
is $\frac{\epsilon}{2}$-lc, and we may run a $(K_X+\Gamma)$-MMP, $\phi: X\dashrightarrow Y$, such that $\phi$ is a birational map, $K_Y+\phi_{*}\Gamma$ is semiample and $(Y, \phi_{*}\Gamma)$ is still $\frac{\epsilon}{2}$-lc.
\medskip

Now $K_Y+\phi_{*}\Gamma$ defines a contraction $f: Y\rightarrow Z$. Let $F$ be a general fiber of $f$, we have $\dim(F)>0$. By restricting to $F$, we get
\begin{equation}\label{eq: numerically trivial}
K_{F}+\phi_{*}\Gamma|_{F}= K_F + \phi_{*}\Delta_{F}+\sum_{i=1}^m \tau_iH_{Y,i}|_F\equiv0,
\end{equation}
where $H_{Y,i}$ is the strict transform of $H_i'$ on $Y$ for each $i$. 


\medskip

Since $K_F + \phi_{*}\Delta+\sum_{i=1}^m \tau_iH_{Y,i}|_F$ is $\frac{\epsilon}{2}$-lc, according to Theorem \ref{thm: BAB}, $F$ belongs to a bounded family. Hence, we may find a very ample Cartier divisor $M$ on $F$, so that $-K_F\cdot M^{\dim F-1}$ is bounded. Besides, as $H'_i \sim_\Qq H_i$, we have $\phi_*(H'_i) \equiv_{\Qq} \phi_*(H_i)$, and the intersection numbers $d_i:=H_{Y,i}|_F \cdot M^{\dim F-1}= \phi_*(H_i)|_F \cdot M^{\dim F-1}$ is a positive integer. Let $\phi_{*}\Delta|_F=\sum_j a_j\Delta_{F,j}$, where $a_j\in \Ii$, and $\Delta_{F,j}$ are Weil divisors. By intersecting \eqref{eq: numerically trivial} with $M^{\dim F-1}$, we obtain an equation for $\tau_i$,
\begin{equation}\label{eq: relation of tau}
\sum_j a_j b_j+\sum_{i=1}^m \tau_id_i=c,
\end{equation}
where $c = - K_F \cdot M^{\dim F-1}$ is a nonnegative integer with only finite possibilities, $b_j = \Delta_{F,j} \cdot M^{\dim F-1}$ is a positive integer. Since $\tau_i\ge \delta$ for all $i$, and $\Ii$ is finite, $d_i,b_i$ are bounded above and thus only have finite possibilities. Hence there are only finite possibilities for the equations \eqref{eq: relation of tau}. In other words, $(\tau_1,\ldots,\tau_m)$ can only lie on finitely many hyperplanes $L_k \subseteq \Rr^m, k\in I$. By Proposition \ref{prop: polytope}, $P(X,\Delta;H_1,\ldots,H_m)\cap [\epsilon,+\infty)^m$ is a polytope. If $\Theta$ is a facet of $P(X;H_1,\ldots,H_m)\cap [\epsilon,+\infty)^m$ (i.e. an $(m-1)$-dimensional face), then
\[
\Theta \subseteq (\cup_{k \in I} L_k) \cup (\cup_{i=1}^m\{t_i = \delta\}).
\]
By irreducibility, $\Theta \subseteq L_k$ or $\Theta \subseteq \{t_i = \delta\}$, and hence finite possibilities. This shows that the set ${\mathcal P}_{n,m,\Ii,\epsilon,\delta}$ is finite.
\end{proof}
\begin{figure}[h]
      \centering
\begin{tikzpicture}

\draw [help lines,<->] (0,2.8)--(0,0)--(2.5,0);
\draw  (0,2.3)--(0,1.8)--(0.3,1.1)--(1,0.3)--(1.8,0)--(2.3,0);
\draw [thick] (0.3,1.1)--(1,0.3);

\draw  (0,2.4)--(0,2.1)--(0.4,1.3)--(1.2,0.5)--(2,0)--(2.4,0);
\draw [thick] (0.4,1.3)--(1.2,0.5);

\draw  (0,2.6)--(0,2.4)--(1,0.9)--(2.3,0)--(2.5,0);
\draw [thick] (0,2.4)--(1,0.9);

\draw [help lines, ->] (1.4,-.2)--(1.4,.25);
\draw [help lines, ->] (1.8,-.2)--(1.8,.25);
\draw [help lines, ->] (2.2,-.2)--(2.2,.25);

\node [below] at (0.6,0.75) {\tiny$\Theta_1$};
\node [below] at (0.8,1) {\tiny$\Theta_2$};
\node at (0.55, 1.35) {\tiny$\Theta_3$};

\node at (1.5,1.5) {$\cdot$};
\node at (1.6,1.6) {$\cdot$};
\node at (1.7,1.7) {$\cdot$};

\node at (1.4,-.3){\tiny$P_1$};
\node at (1.8,-.3){\tiny$P_2$};
\node at (2.2,-.3){\tiny$P_3$};

\node[right] at (2.5,0) {\footnotesize$t_1$};
\node[above] at (0,2.8) {\footnotesize$t_2$};

\end{tikzpicture}
\caption{A decreasing sequence of PE-polytopes}
\label{fig: DCC}
\end{figure}
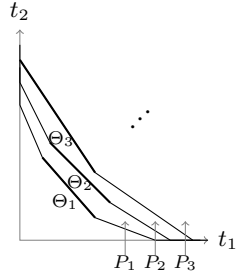

\begin{proof}[Proof of Theorem \ref{thm: DCC}]
By Lemma \ref{lem:qfactorial}, we can assume that $X$ is $\Qq$-factorial. Without loss of generality, we can assume that $K_X$ is not pseudo-effective.

\medskip

  Choose arbitrary $\tau=(\tau_1\ldots,\tau_m)$ on the boundary of the PE-polytope $P(X,\Delta;H_1,\ldots,H_m)$ which is in the interior of $\Rr_{\geq0}^m$. We then proceed the same way as in the proof of Theorem \ref{thm: OneDCCfinite}. We can find $\Qq$-divisors $H'_i\sim_{\Qq}H_i$, such that
$$(X,\Gamma)\coloneqq(X,\Delta+\sum_{i=1}^m \tau_iH'_i)$$
is $\frac{\epsilon}{2}$-lc, and we may run a $(K_X+\Gamma)$-MMP, $\phi: X\dashrightarrow Y$, such that $K_Y+\phi_{*}\Gamma$ is semiample. $K_Y+\phi_{*}\Gamma$ defines a contraction $f: Y\rightarrow Z$. By restricting to a general fiber $F$, we get a similar equation as \eqref{eq: numerically trivial}
\begin{equation}\label{eq: numerically trivial 2}
K_{F}+\phi_{*}\Gamma|_{F}= K_F +\sum_j a_j\Delta_{F,j}+ \sum_{i=1}^m \tau_iH_{Y,i}|_F\equiv0.
\end{equation} Here we write $\phi_{*}\Delta|_{F}=\sum_j a_j\Delta_{F,j}$, where $\Delta_{F,j}$ is a Weil divisor. As $({F},\phi_{*}\Gamma|_{F})$ is $\frac{\epsilon}{2}$-lc, it belongs to a bounded family by Theorem \ref{thm: BAB}. By taking a very ample divisor $M$, and intersecting $M^{\dim F-1}$ with \eqref{eq: numerically trivial 2}, we get

\begin{equation}\label{eq: relation on tau 2}
\sum_j a_j b_j+\sum_{i=1}^m \tau_id_i=c,
\end{equation}
where $c = - K_F \cdot M^{\dim F-1}$ is a nonnegative integer with only finite possibilities, $b_j = \Delta_{Y,j} \cdot M^{\dim F-1}$, $d_i=H_{Y,i}|_F\cdot M^{\dim F-1}$ are nonnegative integers.

\medskip

Now, all $a_jb_j$ form a DCC set, and the finite summations $\sum_j a_j b_j \in \sum \Ii$ also form a DCC set. 

\medskip 

Set $P^{(k)} = P(X^{(k)}, \Delta^{(k)}; H_{1}^{(k)}, \ldots, H_{m}^{(k)})$. Suppose $P^{(1)} \supsetneq P^{(2)} \supsetneq \cdots \supsetneq P^{(k)} \supsetneq  \cdots $ is a sequence of decreasing polytopes. Then there are \emph{countably} many linear equations $\{L^{(s)}\}_{s\in\Nn}$, where 
\begin{equation}\label{eq: L}
L^{(s)}\coloneqq \sum_{i=1}^m t_id_i^{(s)}+\sum_j a_j^{(s)} b_j^{(s)} -c^{(s)},
\end{equation}
where $d_i^{(s)},b_j^{(s)}\in \Nn$, $a_j^{(s)}\in \Ii$, and $c^{(s)}$ belongs to a finite set. If $\Theta^{(k)}$ is a facet of $P^{(k)}$, then $\Theta^{(k)}$ must lie on the \emph{countably} union of hyperplanes
 \[
 (\cup_{s\in \Nn} \{L^{(s)} =0\}) \cup (\cup_{i=1}^m\{t_i=0\}),
 \]
 where $\{t_i=0\}$ is the $i$-th coordinate hyperplane. If $\Theta^{(k)}$ is not contained in any of  $\{L^{(s)} =0\}$ nor $\{t_i=0\}$, then their intersections on $\Theta^{(k)}$ are measure zero sets. Thus, their countable unions is still of measure zero. This is impossible and hence $\Theta^{(k)}$ must be contained in one of $\{L^{(s)} =0\}$ or $\{t_i=0\}$.

\medskip

 Since $\{P^{(k)}\}_{k\in\Nn}$ is a strictly decreasing sequence, there must exist a sequence $\{\Theta^{(k)}\}_{k\in \Nn}$ such that $\Theta^{(k)} \subseteq P^{(k)}$ is a facet, $\Theta^{(k)} \subseteq \{ t \in \Rr_{\geq 0}^m\mid L^{(k)}(t)=0\}$, and 
$$\{ t \in \Rr_{\geq 0}^m\mid L^{(k)}(t)=0\}_{k\in \Nn}$$
are different sets (c.f. Figure \ref{fig: DCC}).
Using the DCC property, by passing to a subsequence, we can assume that, in \eqref{eq: L}, $c^{(k)}=c$ is fixed for all $k$, the sequence $\{\sum_j a^{(k)}_j b^{(k)}_j\}_{k\in\Nn}$ is non-decreasing, and the sequence $\{n^{(k)}_i\}_{k\in\Nn}$ is non-decreasing for each $i$.

Now, the set
\begin{equation}\label{eq: region increases}
\{ t \in \Rr_{\geq 0}^m\mid L^{(k)}(t) \geq 0\}_{k\in \Nn}
\end{equation} is strictly increasing. However, by assumption, we have
\[
\{ t \in \Rr_{\geq 0}^m\mid L^{(k)}(t) \geq 0\} \supseteq P^{(k)} \supsetneq P^{(k+1)} \supseteq \Theta^{(k+1)}.
\]
This is impossible as there exists $\theta \in \Theta^{(k+1)}$ lies on $L^{(k+1)}=0$ which is not contained in $\{ t \in \Rr_{\geq 0}^m\mid L^{(k)}(t) \geq 0\}$ by \eqref{eq: region increases}.
\end{proof}
\bibliographystyle{alpha}

\begin{thebibliography}{BCHM10}

\bibitem[BCHM10]{BCHM10}
Caucher Birkar, Paolo Cascini, Christopher~D. Hacon, and James McKernan.
\newblock Existence of minimal models for varieties of log general type.
\newblock {\em J. Amer. Math. Soc.}, 23(2):405--468, 2010.

\bibitem[Bir16a]{Bir16a}
Caucher Birkar.
\newblock Anti--pluricanonical systems on Fano varieties.
\newblock {\em arXiv:math/1603.05765}, 2016.

\bibitem[Bir16b]{Bir16b}
Caucher Birkar.
\newblock Singularities of linear systems and boundedness of Fano varieties.
\newblock {\em arXiv:math/1609.05543}, 2016.

\bibitem[Bor97]{Bor97}
Alexandr Borisov.
\newblock Minimal discrepancies of toric singularities.
\newblock {\em Manuscripta Math.}, 92(1):33--45, 1997.

\bibitem[BZ16]{BZ16}
Caucher Birkar and De-Qi Zhang.
\newblock Effectivity of {I}itaka fibrations and pluricanonical systems of
  polarized pairs.
\newblock {\em Publ. Math. Inst. Hautes \'Etudes Sci.}, 123:283--331, 2016.

\bibitem[DC16]{DiC16}
Gabriele Di~Cerbo.
\newblock On {F}ujita's log spectrum conjecture.
\newblock {\em Math. Ann.}, 366(1-2):447--457, 2016.

\bibitem[DC17]{DiC17}
Gabriele Di~Cerbo.
\newblock On {F}ujita's spectrum conjecture.
\newblock {\em Adv. Math.}, 311:238--248, 2017.

\bibitem[Fuj92]{Fujita92}
Takao Fujita.
\newblock On {K}odaira energy and adjoint reduction of polarized manifolds.
\newblock {\em Manuscripta Math.}, 76(1):59--84, 1992.

\bibitem[Fuj96]{Fujita96}
Takao Fujita.
\newblock On {K}odaira energy of polarized log varieties.
\newblock {\em J. Math. Soc. Japan}, 48(1):1--12, 1996.

\bibitem[Ful93]{Ful93}
William Fulton.
\newblock {\em Introduction to toric varieties}, volume 131 of {\em Annals of
  Mathematics Studies}.
\newblock Princeton University Press, Princeton, NJ, 1993.
\newblock The William H. Roever Lectures in Geometry.

\bibitem[HJ16]{HJ16}
Christopher~D. Hacon and Chen Jiang.
\newblock On Fujita invariants of subvariaties of a uniruled variety.
\newblock {\em Algebraic Geomerty}, 4(3):304--310, 2017.

\bibitem[HMX13]{HMX13}
Christopher~D. Hacon, James M\textsuperscript{c}Kernan, and Chenyang Xu.
\newblock On the birational automorphisms of varieties of general type.
\newblock {\em Ann. of Math. (2)}, 177(3):1077--1111, 2013.

\bibitem[HMX14]{HMX14}
Christopher~D. Hacon, James McKernan, and Chenyang Xu.
\newblock A{CC} for log canonical thresholds.
\newblock {\em Ann. of Math. (2)}, 180(2):523--571, 2014.

\bibitem[KM98]{KM98}
J\'anos Koll\'ar and Shigefumi Mori.
\newblock {\em Birational geometry of algebraic varieties}, volume 134 of {\em
  Cambridge Tracts in Mathematics}.
\newblock Cambridge University Press, Cambridge, 1998.
\newblock With the collaboration of C. H. Clemens and A. Corti, Translated from
  the 1998 Japanese original.

\bibitem[LTT14]{LTT}
B.~Lehmann, S.~Tanimoto, and Y.~Tschinkel.
\newblock Balanced line bundles on Fano varieties.
\newblock {\em J. Reine Angew. Math.}, to appear.

\end{thebibliography}

\end{document}